\documentclass[preprint,10pt,a4paper,final]{elsarticle}
\usepackage{amsmath,amssymb,a4wide}
\usepackage{color,graphicx,epstopdf}
\usepackage[a4paper]{geometry}
\usepackage{lineno}

\newcommand{\bv}{\boldsymbol v}

\newcommand{\bx}{\boldsymbol x}

\newcommand{\bU}{\boldsymbol U}

\newcommand{\bS}{\boldsymbol S}

\newcommand{\linenomathpatch}[1]{%
  \cspreto{#1}{\linenomath}%
  \cspreto{#1*}{\linenomath}%
  \csappto{end#1}{\endlinenomath}%
  \csappto{end#1*}{\endlinenomath}%
}
\linenomathpatch{equation}
\linenomathpatch{gather}
\linenomathpatch{multline}
\linenomathpatch{align}
\linenomathpatch{alignat}
\linenomathpatch{flalign}

\newtheorem{Theorem}{Theorem}
\newtheorem{lema}{Lemma}
\newcounter{remark}
\def\theremark {\arabic{remark}}

\newtheorem{Proof}{Proof}

\newenvironment{proof}{\begin{Proof}\rm}{\hfill $\Box$ \end{Proof}}
\allowdisplaybreaks

\usepackage{hyperref}

\date{\today}
\begin{document}
\begin{frontmatter}
\title{Second order error bounds for POD-ROM methods based on first order divided differences }

\author[sevilla]{Bosco Garc\'{\i}a-Archilla}
\ead{bosco@esi.us.es; research is supported by
Spanish MCINYU under grants PGC2018-096265-B-I00 and PID2019-104141GB-I00}
\author[wias,fu]{Volker John}
\ead{john@wias-berlin.de, ORCID 0000-0002-2711-4409}
\author[madrid]{Julia Novo\texorpdfstring{\corref{cor}}{}}
\ead{julia.novo@uam.es; research is supported
by Spanish MINECO
under grants PID2019-104141GB-I00 and VA169P20}

\cortext[cor]{Corresponding author}

\address[sevilla]{Departamento de Matem\'atica Aplicada II, Universidad de Sevilla, Sevilla, Spain}

\address[wias]{Weierstrass Institute for Applied Analysis and Stochastics (WIAS), Mohrenstr.
39, 10117 Berlin, Germany}

\address[fu]{Freie Universit\"at Berlin, Dept. of Mathematics and
Computer Science, Arnimallee 6, 14195 Berlin, Germany}

\address[madrid]{Departamento de
Matem\'aticas, Universidad Aut\'onoma de Madrid, Spain}

\begin{abstract} This note proves, for simplicity for the heat equation, that using BDF2 as time 
stepping scheme in POD-ROM methods with snapshots based on difference quotients gives both the optimal second order error bound in time
and pointwise estimates. 
\end{abstract}

\end{frontmatter}

\section{Introduction}
Most numerical methods using reduced order models based on proper orthogonal decomposition (POD-ROM methods) apply
basis functions based on the snapshots (or values at different times) of the full order model (FOM). Recently, it has been shown that adding their first divided differences to the snapshots, or even using only these divided differences to obtain the basis functions, allows for pointwise-in-time error bounds \cite{Lo-Sin,nos_der_pod,samu_et_al,novo_rubino}. However, all pointwise-in-time error bounds in the literature are only first order with respect to time. 

Although the first divided differences are only first order approximations to the time derivatives of the snapshots, we show in this note that for POD-ROM methods based only on them
 it is possible to obtain pointwise-in-time second order error bounds if the two step backward differentiation formula (BDF2) is used to integrate the POD-ROM equations. This result is a theoretical support for the 
 observation that second order methods allow for larger step sizes than first order ones without spoiling the error, thus resulting in more efficient POD-ROM simulations.

\section{Model problem and proper orthogonal decomposition}

Throughout this note, standard notations for Sobolev spaces and their norms will be used. 
As a model problem problem, we consider the heat equation
\begin{equation*}
\begin{array}{rcll}
\partial_t u(t,\bx)-\nu\Delta u(t,\bx)&=&f(t,\bx), &(t,\bx)\in (0,T]\times \Omega,\\
u(t,\bx)&=&0,  &(t,\bx)\in (0,T]\times \partial \Omega,\\
u(0,\bx)&= &u^0(\bx), &\bx\in \Omega,
\end{array}
\end{equation*}
in a bounded domain $\Omega\subset {\Bbb R}^d$, $d\in \{2,3\}$. Let $C_p$ be the constant in the Poincar\'e inequality
\begin{equation}\label{poincare}
\|v\|_0\le C_p\|\nabla v\|_0,\quad v\in H_0^1(\Omega).
\end{equation}

Let us denote by $X_h^l$ a finite element method based on piece-wise continuous polynomials of degree $l$ that satisfies the homogeneous Dirichlet boundary conditions.
The semi-discrete Galerkin approximation, the FOM, consists in finding $u_h : [0,T]\to X_h^l$ such that
\[
(\partial_t u_h,v_h)+\nu(\nabla u_h,\nabla v_h)=(f,v_h),\quad \forall\  v_h\in X_h^l.
\]
The following error estimation is well-known:
\begin{equation}\label{cota_gal}
\max_{0\le s\le T}\left(\|(u-u_h)(s)\|_0+h\|(u-u_h)(s)\|_1\right)\le C(u)h^{l+1}.
\end{equation}

Fix $T>0$ and set $\Delta t=T/M$. Let $t^n=n\Delta t$, $n=0,\ldots,M$, $N=M+1$, and define the space
\[
\bU = {\rm span}\left\{\sqrt{N}w_0,\tau\frac{u_h(t^1)-u_h(t^0)}{\Delta t},\tau\frac{u_h(t^2)-u_h(t^1)}{\Delta t}\ldots,\tau\frac{u_h(t^M)-u_h(t^{M-1})}{\Delta t}\right\},\\
\]
where $w_0$ is either $w_0=u_h(t^0)$ or $w_0=\overline u_h=\sum_{j=0}^Mu_h(t^j)/(M+1)$,
and $\tau$ is a time scale to make the snapshots dimensionally correct.
Denote
$
\bU={\rm span}\{y_h^1,y_h^2,\ldots,y_h^N\}.
$
Let $X$ be either $X=L^2(\Omega)$ or $X=H_0^1(\Omega)$, and denote the correlation matrix by $K=((k_{i,j}))\in {\Bbb R}^{N\times N}$ with
$ k_{i,j}=(y_h^i,y_h^j)_X/N$,  $i,j=1,\ldots,N$,
and $(\cdot,\cdot)_X$ being the inner product in $X$. We denote by $\lambda_1\ge \lambda_2\ldots\ge \lambda_d>0$ the positive eigenvalues of $K$ and
by $\bv_1,\ldots,\bv_d\in {\Bbb R}^N$ the associated eigenvectors. The orthonormal POD basis functions of $\bU$ are
given by $\varphi_k=  (\sum_{j=1}^Nv_k^k y_h^j)/(\sqrt{N}\sqrt{\lambda_k})$,
where $v_k^j$ is the $j$-th component of $\bv_k$.
For any $1\le r\le d$  denote by
$
\bU^r={\rm span}\left\{\varphi_1,\varphi_2,\ldots,\varphi_r\right\},
$
and denote by $P^r : X_h^l\to \bU^r$ the $X$-orthogonal projection onto $\bU^r$. Then, it holds
\begin{equation}\label{cota_pod}
\frac{1}{N}\sum_{j=1}^N\|y_h^j-P^r y_h^j\|_X^2=\sum_{k={r+1}}^d\lambda_k.
\end{equation}
The stiffness matrix of the POD basis is given by $S=((s_{i,j}))\in {\Bbb R}^{d\times d}$, with $s_{i,j}=(\nabla \varphi_i,\nabla \varphi_j)_X$. 
If $X=L^2(\Omega)$ the following inequality holds for all $v\in \bU$, see \cite[Lemma~2]{Ku-Vol},
\begin{equation}\label{inv_POD}
\|\nabla v\|_0\le \sqrt{\|S\|_2}\|v\|_0.
\end{equation}

\section{Error analysis}


Let us denote by $D^1 v^n=(v^n-v^{n-1})/\Delta t$ and by $D^2 v^n=((3/2)v^n-2v^{n-1}+(1/2)v^{n-2})/\Delta t$, then the POD-ROM method is defined in the following way: Find
$u_r^n\in \bU^r$ such that 
\begin{equation*}
\left({Du_r^n},v\right)+\nu(\nabla u_r^n,\nabla v)=(f^n,v),\quad\forall\  v\in \bU^r,
\end{equation*}
where $D=D^1$ for $n=1$  and $D=D^2$ for $2\le n\le M$.

\begin{lema}\label{le:diferen}
Let $T>0$, 
let $X$ be a {Banach} space, $z^n=z(t^n)\in X$, then
\begin{eqnarray}\label{eq:dif_losin}
\max_{0\le k\le {M} }\|z^k\|_X^2  & \le &  2\|z^0\|_X^2+\frac{ 2T^2}{M}\sum_{n=1}^M\left\|D^1z_n\right\|_X^2, \\
\label{eq:dif_losin_mean}
\max_{0\le k\le {M} }\|z^k\|_X^2 & \le&   2\|\overline z\|_X^2+\frac{ 8T^2}{M}\sum_{n=1}^M\left\|D^1z_n\right\|_X^2, \quad \mbox{with}\quad \overline z=\sum_{j=0}^Mz^j/(M+1).
 \end{eqnarray}
\end{lema}

\begin{proof} The proof of \eqref{eq:dif_losin} can be found in \cite[Lemma~3.3]{Lo-Sin}. 
For proving \eqref{eq:dif_losin_mean}, we observe that
\begin{equation}\label{zhal}
z^k=z^0+\Delta t \sum_{n=1}^k {D^1z^n},\quad
\overline z =z^0+\frac{1}{M+1}\left(\Delta t{D^1 z^1}+\ldots+\Delta t \sum_{n=1}^M{D^1 z^n}\right).
\end{equation}
Taking norms yields $\| z^k\|_X\le \|z^0\|_X+\Delta t \sum_{n=1}^M  \left\|{D^1z^n}\right\|_X$
and $\| z^0\|_X\le \|\overline z\|_X+\Delta t \sum_{n=1}^M  \left\|{D^1z^n}\right\|_X$, 
so that
\[
\|z^k\|_X\le  \|\overline z\|_X+2\Delta t \sum_{n=1}^M  \left\|{D^1z^n}\right\|_X\le \|\overline z\|_X+
2T^{1/2}(\Delta t)^{1/2}
\left(\sum_{n=1}^M\left\|D^1 z^n\right\|_X^2\right)^{1/2},
\]
from which we reach \eqref{eq:dif_losin_mean}.
\end{proof}

In the sequel we define $\tilde C=1$ if $w_0=u_h(t^0)$ and $\tilde C=4$ if~$w_0=\overline u_h$, and $C_X=1$ if~$X=L^2(\Omega)$
and $C_X=C_p^2$ f~$X=H^1_0(\Omega)$.

\begin{lema} The following bound holds
\begin{equation}\label{max_dif}
\max_{0\le n\le M}\|u_h^n-P^ru_h^n\|_0^2\le \left(2+4\tilde C \frac{T^2}{\tau^2}\right)C_X\sum_{k={r+1}}^d\lambda_k.
\end{equation}
\end{lema}
\begin{proof}
Taking $z=u_h-P^r u_h$ in \eqref{eq:dif_losin} or \eqref{eq:dif_losin_mean}, depending on the selection of the 
first element in $\bU$, and applying \eqref{cota_pod} and $N\le2M$,  we reach \eqref{max_dif}.
\end{proof}

\begin{lema}\label{le:cgen} Let $\{z^n\}_{n=0}^N \in \bU^r$  and $\{\tau_1^n\}_{n=1}^N, 
\{\tau_2^n\}_{n=1}^N \in X_h^l$ satisfying
\begin{equation}\label{error_equation00}
\left({D z^n},v\right)+\nu(\nabla z^n,\nabla v)=\left(\tau_1^n,v\right)+\nu(\nabla \tau_2^n,\nabla v),\quad \forall\  v\in \bU^r,
\end{equation}
where $D=D^1$ for $n=1$ and $D=D^2$ for $2\le n\le M$.
Then, it holds for $\Delta t<T/4$ and $n\ge 1$
\begin{eqnarray}\label{eq:cgen2}
\|z^n\|_0^2+2\nu\sum_{j=1}^n\Delta t \|\nabla z^j\|_0^2
&\le& e^{4}\left(17\|z^0\|_0^2+28(\Delta t)^2\left\|\tau_1^1\right\|_0^2+2\Delta t T\sum_{j=2}^N \left\| \tau_1^n\right\|_0^2\right.\nonumber\\
&&\quad \left.+14\nu\Delta t \|\nabla \tau_2^1\|_0^2+2\nu \Delta t\sum_{j=2}^N \left\| \nabla\tau_2^n\right\|_0^2\right).
\end{eqnarray}
\end{lema} 
\begin{proof}
We take $v=\Delta t z^n$ in~\eqref{error_equation00}. 
If $n=1$ then $D=D^1$ and Young's inequality yields
\begin{equation}
\label{eq:cgenaux00}
\frac{1}{2}\|z^1\|_0^2-\frac{1}{2}\|z^{0}\|_0^2+{\nu}\Delta t \|\nabla z^1\|_0^2 \le \Delta t(\tau_1^1,z^1) + \nu\Delta t(\nabla \tau_2^1,\nabla z^1).
\end{equation}
For $n\ge 2$ then $D=D^2$ and one gets
\[
\frac{1}{4}\|z^n\|_0^2+\frac{1}{4}\|\hat z^n\|_0^2-\frac{1}{4}\|z^{n-1}\|_0^2-\frac{1}{4}\|\hat z^{n-1}\|_0^2
+\nu\Delta t\|\nabla z^n\|_0^2 \le \Delta t(\tau_1^n,z^n) +
 \nu \Delta t(\nabla\tau_2^n,\nabla z^n),
\]
where $\hat z^n = 2z^n - z^{n-1}$. The Cauchy--Schwarz and Young inequality give
\begin{equation}
\label{eq:cgenaux0}
 \Delta t(\tau_1^n,z^n) + \nu\Delta t(\nabla \tau_2^n,\nabla z^n)\le \frac{\Delta t}{2T}\left\|z^n\right\|_0^2 + \frac{T}{2} \left\|\tau_1^n\right\|_0^2 +\Delta t\frac{\nu}{2}\left\|\nabla z^n\right\|_0^2  +\Delta t\frac{\nu}{2}\left\|\nabla \tau_2^n\right\|_0^2.
\end{equation}
Multiplying by~$4$, applying~\eqref{eq:cgenaux0}, and summing from~$2$ to~$n$, one gets 
\begin{equation}
\label{eq:cgenaux02}
\|z^n\|_0^2+2\nu\sum_{j=2}^n\Delta t \|\nabla z^j\|_0^2 \le \|z^1\|_0^2+\|\hat z^1\|_0^2+2\sum_{j=2}^n\frac{\Delta t}{T}\|z^j\|_0^2
+2T\sum_{j=2}^n\Delta t \|\tau_1^j\|_0^2
 +2\nu\sum_{j=2}^n\Delta t \|\nabla \tau_2^j\|_0^2.
\end{equation}
Young's inequality yields
$
\|\hat z^1\|_0^2 \le 6\|z^1\|_0^2 + 3\| z^0\|^2, 
$
so that 
$
\|z^1\|_0^2+\|\hat z^1\|_0^2 \le 7\|z^1\|_0^2 + 3\| z^0\|^2.
$
Using again Young's inequality gives 
\[
 \Delta t(\tau_1^1,z^1) + \nu\Delta t(\nabla \tau_2^1,\nabla z^1)\le \frac{1}{4}\left\|z^1\right\|_0^2 + (\Delta t)^2\left\|\tau_1^1\right\|_0^2 +\Delta t\frac{\nu}{2}\left\|\nabla z^1\right\|_0^2  +\Delta t\frac{\nu}{2}\left\|\nabla \tau_2^1\right\|_0^2,
\]
so that we obtain from \eqref{eq:cgenaux00}
\[
\left\| z^1\right\|_0^2 +2\nu\Delta t\left\| \nabla z^1\right\|_0^2 \le 2\left\| z^0\right\|_0^2 + 4  (\Delta t)^2\left\|\tau_1^1\right\|_0^2 + 2 \Delta t \nu 
\left\|\nabla \tau_2^1\right\|_0^2.
\]
Together with \eqref{eq:cgenaux02}, it follows that for $n\ge 1$
\begin{eqnarray*}
\|z^n\|_0^2+2\nu\sum_{j=1}^n\Delta t \|\nabla z^j\|_0^2 &\le& 17\|z^0\|_0^2+28 (\Delta t)^2\left\|\tau_1^1\right\|_0^2+2\sum_{j=2}^n\frac{\Delta t}{T}\|z^j\|_0^2
+2T\sum_{j=2}^n\Delta t \|\tau_1^j\|_0^2
\nonumber\\
&& +14\Delta t \nu \|\nabla \tau_2^1\|_0^2+2\nu\sum_{j=1}^n\Delta t \|\nabla \tau_2^j\|_0^2,
\end{eqnarray*}
from where~\eqref{eq:cgen2} follows by applying Gronwall's Lemma \cite[Lemma 5.1]{hey4} for $\Delta t\le T/4$.
\end{proof}

Let $X=L^2(\Omega)$ and 
let us denote by $e_r^n=u_r^n-P^r u_h^n$ and by $\eta_h^n=P^r u_h^n-u_h^n$. Arguing as in the proof of 
\cite[Theorem~4.6]{nos_der_pod}, one gets
\begin{equation}\label{error_equation}
\left({D e_r^n},v\right)+\nu(\nabla e_r^n,\nabla v)=\left(\partial_t u_h^n-{Du_h^n},v\right)-\nu(\nabla \eta_h^n,\nabla v),\quad \forall\  v\in \bU^r.
\end{equation}

\begin{lema}\label{le:consistency} The following bounds hold
\begin{eqnarray}
\label{trun12}
\left\|\partial_t u_h^1-{D^1u_h^1}\right\|_j \!\!\! & \le & \!\!\! \frac{\Delta t}{2} \max_{0\le t\le t_1} \left\| \partial_{tt} u_h\right\|_j,\quad j=0,1,
\\
\label{trun_BDF2}
\left\|\partial_t u_h^n-{D^2u_h^n}\right\|_j  \!\!\! &\le& \!\!\! \sqrt{5} (\Delta t)^{3/2} \biggl(\int_{t_{n-2}}^{t_n} \left\|\partial_{ttt} u_h(t)\right\|_j^2\,dt\biggr)^{1/2},\quad
n=2,\ldots,N, \  j=0,1.
\end{eqnarray}
\end{lema}
\begin{proof}
For $D=D^1$, \eqref{trun12} follows easily from
\[
\partial_t u_h^n-{Du_h^n}=\frac{1}{\Delta t} \int_{t_{n-1}}^{t_n}(\partial_t u_h(t_{n})-\partial_t u_h(s))\,ds =\frac{1}{\Delta t} \int_{t_{n-1}}^{t_n}\biggl(\int_s^{t_n} \partial_{tt} u_h(t)\,dt\biggr)\,ds.
\]
For $D=D^2$, Taylor series expansion with integral reminder reveals that
\[
\partial_t u_h^n-{Du_h^n}=\frac{1}{\Delta t} \int_{t_{n-2}}^{t_n}\left(2(t-t_{n-1})_{+}^2 -\frac{1}{2}(t-t_{n-2})^2\right)\partial_{ttt} u_h\,dt,
\]
where $x_{+}=\max(0,x)$,  for $x\in{\mathbb R}$. Then, a straightforward calculation shows that
\[
\left\|\partial_t u_h^n-{Du_h^n}\right\|_j \le \left(\frac{2}{\sqrt{5}} + \frac{2\sqrt{2}}{\sqrt{5}}\right) (\Delta t)^{3/2} \biggl(\int_{t_{n-2}}^{t_n} \left\|\partial_{ttt} u_h(t)\right\|_j^2\,dt\biggr)^{1/2},
\]
and then \eqref{trun_BDF2} follows by noticing that $2+2\sqrt{2} < 5$.
\end{proof}

\begin{lema} 
Let $X=L^2(\Omega)$. It holds
\begin{equation}\label{eta_u2}
\nu\sum_{j=1}^n\Delta t \|\nabla \eta_h^j\|_0^2\le \nu T\|S\|_2\left(2+4\tilde C \frac{T^2}{\tau^2}\right)\sum_{k={r+1}}^d\lambda_k.
\end{equation}
\end{lema}

\begin{proof}
The proof of \eqref{eta_u2} follows easily by applying \eqref{inv_POD}  and \eqref{max_dif}. 
\end{proof}

\begin{Theorem}[Bound for $X=L^2(\Omega)$]\label{thm:L2_bound}
Let $X=L^2(\Omega)$, then it holds  for  $\Delta t\le T/4$ 
\begin{eqnarray}\label{cota_u2_l2_d2}
\max_{1\le n\le M}\|u_r^n-u^n\|_0^2
 \!\!&\le &\!\!
60e^{4}\left(\|e_r^0\|_0^2+(\Delta t)^4 \max_{0\le s\le \Delta t}\|\partial_{tt} u_h(s)\|_0^2+T(\Delta t)^4\int_0^T\|\partial_{ttt} u_h(s)\|_0^2 \ ds\right)
\nonumber\\
&&{}+3(1+14T\nu e^4\|S\|_2) \left(2+4\tilde C \frac{T^2}{\tau^2}\right)\sum_{k={r+1}}^d\lambda_k+3C(u)^2 h^{2(l+1)}.
\end{eqnarray}
\end{Theorem}
\begin{proof}
From~\eqref{error_equation} and \eqref{eq:cgen2}, applying~\eqref{trun12}, \eqref{trun_BDF2} (noting that most integrals over time intervals $[t_{j-1},t_j]$
appear twice when summing over $n$), and  \eqref{eta_u2},  we obtain
\begin{eqnarray*}
\lefteqn{\|e_r^n\|_0^2+\nu\sum_{j=1}^n\Delta t \|\nabla e_r^j\|_0^2
\le e^{4}\left(17\|e_r^0\|_0^2+7(\Delta t)^4 \max_{0\le s\le \Delta t}\|\partial_{tt} u_h(s)\|_0^2 \right.}\nonumber\\
&&\left.+20T(\Delta t)^4\int_0^T\|\partial_{ttt} u_h(s)\|_0^2 ds
+14 \nu T\|S\|_2\left(2+4\tilde C \frac{T^2}{\tau^2}\right)\sum_{k={r+1}}^d\lambda_k \right).
\end{eqnarray*}
To simplify, we replace the factors $17$ and $7$ by $20$.
To finish the proof, apply the decomposition $u_r^n-u^n=(u_r^n-P_h u_h^n)+(P_h u_h^n-u_h^n)+(u_h^n-u^n)$,
followed by \eqref{max_dif} and \eqref{cota_gal}.
\end{proof}

Let $X=H_0^1(\Omega)$. Arguing as in the proof of \cite[Theorem 4.1]{nos_der_pod} yields
\begin{equation*}
\left({D e_r^n},v\right)+\nu(\nabla e_r^n,\nabla v)=\left(\partial_t u_h^n-P_r\left({D u_h^n}\right),v\right),\quad  \forall v\in \bU^r.
\end{equation*}
Applying Lemma~\ref{le:cgen} with $z^n=e_r^n$, $\tau_1^n = \partial_t u_h^n-P_rD u_h^n$ and~$\tau_2=0$ we get
\begin{equation}\label{eq:H100}
\|e_r^n\|_0^2+2\nu\sum_{j=1}^n\Delta t \|\nabla e_r^j\|_0^2
\le e^{4}\left(17\|e_r^0\|_0^2+28(\Delta t)^2\left\|\tau_1^1\right\|_0^2+2\Delta t T\sum_{j=2}^N \left\| \tau_1^n\right\|_0^2\right).
\end{equation}

\begin{Theorem}[Bound for $X=H_0^1(\Omega)$] Let $X=H_0^1(\Omega)$, 
$\Delta t\le T/4$, and $
C_1=4e^4\left(10+ \Delta t /T\right)+ 2+4\tilde C$. Then it holds  
\begin{eqnarray}\label{cota_u2_l2}
\max_{1\le n\le M}\|u_r^n-u^n\|_0^2 \!\!\! & \le  & \!\!\!
60e^{4}\left(\|e_r^0\|_0^2+ (\Delta t)^4 \max_{0\le s\le \Delta t}\|\partial_{tt} u_h(s)\|_0^2+2(\Delta t)^4\int_0^T\|\partial_{ttt} u_h(t)\|_0^2 \ dt \right)\nonumber\\
&&\quad +3C_1C_p^2\left(\frac{T}{\tau}\right)^2\sum_{j={r+1}}^d \lambda_k+3C^2(u) h^{2(l+1)}.
\end{eqnarray}
\end{Theorem}
\begin{proof}
The last two terms on the right-hand side of  \eqref{eq:H100} are bounded by the triangle inequality
\begin{equation}\label{muno}
\left\|\tau_1^n\right\|_0^2=\left\|\partial_t u_h^n-P^r\left({Du_h^n}\right)\right\|_0^2
\le 2\left\|\partial_t u_h^n-\left({D u_h^n}\right)\right\|_0^2
+2 \left\|(I-P^r)\left({Du_h^n}\right)\right\|_0^2.
\end{equation}
For $n=1$, the first term is bounded by \eqref{trun12} and the second one by \eqref{poincare} and \eqref{cota_pod},
giving
\begin{equation*}
(\Delta t)^2\left\|\tau_1^1\right\|_0^2\le \frac{(\Delta t)^4}{2} \max_{0\le s\le \Delta t}\|\partial_{tt} u_h(s)\|_0^2+\frac{4T}{\tau^2}C_p^2\Delta t \sum_{k={r+1}}^d\lambda_k.
\end{equation*}
For $n\ge 2$, the first term of \eqref{muno} is estimated by \eqref{trun_BDF2}.
To bound the other term observe that 
\[
D^2u_h^n=(3/2)D^1 u^n-(1/2)D^1 u^{n-1},
\]
and, consequently,
\[
2 \left\|(I-P^r)\left({D^2u_h^n}\right)\right\|_0^2 \le \frac{9}{2}
\left\|(I-P^r)\left({D^1u_h^n}\right)\right\|_0^2 + \frac{1}{2}
\left\|(I-P^r)\left({D^1u_h^n}\right)\right\|_0^2,
\]
so that, by using \eqref{poincare} and \eqref{cota_pod}, one obtains
\[
2T\sum_{j=2}^n\Delta t \left\|(I-P^r)\left({D^2u_h^n}\right)\right\|_0^2
\le  10 T\sum_{j=1}^{n}\Delta t \left\|(I-P^r)\left({D^1u_h^n}\right)\right\|_0^2
\nonumber\\
\le 20C_p^2\left(\frac{T}{\tau}\right)^2\sum_{j={r+1}}^d \lambda_k.
\]
Collecting the estimates for $n=1$ and $n\ge 2$ leads to 
\begin{eqnarray*}
\|e_r^n\|_0^2+2\nu\sum_{j=1}^n\Delta t \|\nabla e_r^j\|_0^2
&\le& e^{4}\left(17\|e_r^0\|_0^2+14 (\Delta t)^4 \max_{0\le s\le \Delta t}\|\partial_{tt} u_h(s)\|_0^2\right.\nonumber\\
&&\left.+40(\Delta t)^4\int_0^T\|\partial_{ttt} u_h(t)\|_0^2 \,dt +4\left(10+ \frac{\Delta t}{T}\right)C_p^2\frac{T^2}{\tau^2}  \sum_{j={r+1}}^d \lambda_k\right).
\end{eqnarray*}
Now, the proof is finished in the same way as the proof of Theorem~\ref{thm:L2_bound}.
\end{proof}

Second order error bounds in time of form \eqref{cota_u2_l2_d2} and \eqref{cota_u2_l2} can be derived if 
the finite differences in $\bU$ are replaced with the temporal derivatives $\{\partial_t u_h^n\}_{n=0}^M$, 
with only slight modifications in the analysis. If the set of snapshots is $\{u_h^n\}_{n=0}^M$, then a
second order estimate for $\sum_{j=1}^M \Delta t \|u_r^n-u^n\|_0^2$ can be shown along the lines of the presented analysis but neither pointwise estimates nor optimal estimates in the $H^1$ norm can be obtained.

\bibliographystyle{abbrv}
\bibliography{references}
\end{document}